\documentclass[11pt,reqno]{amsart}
\usepackage{xifthen,setspace,csquotes}
\usepackage{pkgfile}
\newtheorem{example}{Example}
 
\allowdisplaybreaks

\title[Limit Theorems for Monochromatic Stars]{Limit Theorems for Monochromatic Stars}

\author[Bhattacharya]{Bhaswar B. Bhattacharya}
\address{Department of Statistics, University of Pennsylvania, Philadelphia, USA,
{\tt bhaswar@wharton.upenn.edu}}

\author[Mukherjee]{Sumit Mukherjee\textsuperscript{*}}\thanks{\textsuperscript{*}Research partially supported by NSF grant DMS-1712037}
\address{Department of Statistics, Columbia University, New York, USA, {\tt  sm3949@columbia.edu}}

\begin{document}

\begin{abstract} Let $T(K_{1, r}, G_n)$ be the number of monochromatic copies of the $r$-star $K_{1, r}$ in a uniformly random coloring of the vertices of the graph $G_n$. In this paper we  provide a complete characterization of the limiting distribution of $T(K_{1, r}, G_n)$, in the regime where $\E(T(K_{1, r}, G_n))$ is bounded, for any growing sequence of graphs $G_n$. The asymptotic distribution is a sum of mutually independent components, each term of which is a polynomial of a single Poisson random variable of degree at most $r$. Conversely, any limiting distribution of $T(K_{1, r}, G_n)$ has a representation of this form.  Examples and connections to the birthday problem are discussed. 
\end{abstract}


%\spacing{1.125}

\maketitle

\section{Introduction}

Let $G_n$ be a simple labelled undirected graph with vertex set $V(G_n):=\{1,2,\cdots,|V(G_n)|\}$, edge set $E(G_n)$, and adjacency matrix $A(G_n)=\{a_{ij}(G_n),i,j\in  V(G_n)\}$. %, denote by $V(G_n) $ the set of vertices, and by $E(G_n)$ the set of edges of $G_n$, respectively.  
 In a {\it uniformly random $c_n$-coloring of $G_n$}, the vertices of $G_n$ are colored with $c_n$ colors as follows:  
\begin{equation}\P(v\in V(G_n) \text{ is colored with color } a\in \{1, 2, \ldots, c_n\})=\frac{1}{c_n},
\label{eq:uniform}
\end{equation}
independent from the other vertices. An edge $(a, b) \in E(G_n)$ is said to be {\it monochromatic} if $X_a=X_b$, where $X_v$ denotes the color of the vertex $v \in V(G_n)$ in a uniformly random $c_n$-coloring of $G_n$. Denote by  
\begin{align}\label{eq:me}
T(K_2, G_n)=\sum_{1\le u < v\le |V(G_n)|} a_{uv}(G_n) \bm 1\{X_u=X_v\},
\end{align} 
the number of monochromatic edges in $G_n$.

The statistic \eqref{eq:me} arises in several contexts, for example, as the Hamiltonian of the Ising/Potts models on $G_n$ \cite{bmpotts}, in non-parametric two-sample tests \cite{fr}, and the discrete logarithm problem \cite{ghdl}. Moreover, the asymptotics of $T(K_2, G_n)$ is often useful in the study of coincidences \cite{diaconismosteller} as a generalization of the birthday paradox \cite{barbourholstjanson,dasguptasurvey,diaconisholmes,diaconismosteller}: If $G_n$ is a friendship-network graph colored uniformly with $c_n=365$ colors (corresponding to birthdays), then two friends will have the same birthday whenever the corresponding edge in the graph $G_n$ is monochromatic.\footnote{When the underlying graph $G_n=K_n$ is the complete graph $K_n$ on $n$ vertices, this reduces to the classical birthday problem.}  Therefore, $\P(T(K_{2}, G_n)>0)$ is the probability that there are two friends with the same birthday. Note that $\P(T(K_{2}, G_n)>0)=1-\P(T(K_{2}, G_n)=0)=1-\chi_{G_n}(c_n)/c_n^{|V(G_n)|}$, where $\chi_{G_n}(c_n)$ counts the number of proper colorings of $G_n$ using $c_n$ colors. The function $\chi_{G_n}$ is known as the {\it chromatic polynomial} of $G_n$, and is a central object in graph theory \cite{chromaticbook,toft_unsolved,stanley}.

It is well-known that the limiting distribution of $T(K_{2}, G_n)$,  exhibits a universality, that is, $T(K_2, G_n)\dto \dPois(\lambda)$, whenever $\E(T(K_2, G_n))=\frac{|E(G_n)|}{c_n} \rightarrow \lambda$, {\it for any graph sequence $G_n$}.  This was shown by Barbour et al. \cite[Theorem 5.G]{barbourholstjanson}, using the Stein's method for Poisson approximation, for any sequence of deterministic graphs. Recently, Bhattacharya et al. \cite[Theorem 1.1]{BDM} gave a new proof of this result based on the method of moments, which illustrates interesting connections to extremal combinatorics.

For a general graph $H$, define $T(H, G_n)$ to be the number of monochromatic copies of $H$ in $G_n$, where the vertices of $G_n$ are colored uniformly at random with $c_n$ colors as in \eqref{eq:uniform}.   Conditions under which $T(H, G_n)$ is asymptotically Poisson are easy to derive using Stein's method based on dependency graphs \cite{birthdayexchangeability,CDM}. However, the class of possible limiting distributions of $T(H, G_n)$, for a general graph $H$ in the regime where $\E (T(H, G_n))=O(1)$, can be extremely diverse (including mixture and polynomials in Poissons \cite{BDM}), and there is no natural universality, as in the case of edges. Recently, Bhattacharya et al. \cite{BMM} proved the following second-moment phenomenon for the asymptotic Poisson distribution of $T(H,G_n)$, for any connected graph $H$: $T(H,G_n)$ converges to $\dPois(\lambda)$ whenever $\E T(H,G_n) \rightarrow \lambda$ and $\Var T(H, G_n) \rightarrow \lambda$. Moreover, for any graph $H$, $T(H,G_n)$ converges to linear combination of independent Poisson variables, when $G_n$ is a converging sequence of dense graphs \cite{BMdense}. 

However, there is no description of the set of possible limits of $T(H, G_n)$, other than the case of monochromatic edges ($H=K_2$) or dense graphs $G_n$ (where the limits are Poisson or a linear combination of independent Poissons respectively). In this paper, we consider the case of the $r$-star ($H=K_{1, r}$). This arises as a generalization of the birthday problem, for example, with $r=2$ and a friendship network $G_n$, $T(K_{1, 2}, G_n)$ counts the number of triples with the same birthday where someone is friends with the other two. This is especially relevant when $G_n$ has a few influential nodes which have many friends (``superstar" vertices \cite{BSZ}), and we wish to count the number of triple birthday matches with a superstar. 

 In this paper we identity the set of all possible limiting distributions of $T(K_{1, r}, G_n)$, for {\it any graph sequence} $G_n$. We show that the asymptotic distribution of $T(K_{1, r}, G_n)$ is a sum of mutually independent components, each term of which is a polynomial of a single Poisson random variable of degree at most $r$, and, conversely, any limiting distribution of $T(K_{1, r}, G_n)$ has this form.

\subsection{Limiting Distribution for Monochromatic $r$-Stars}

 Let $G_n $ be a simple graph with vertex set $V(G_n)$ and edge set $E(G_n)$. For a fixed graph $H$, denote by $N(H, G_n)$ the number of isomorphic copies of $H$ in $G_n$. Note that $N(K_{1, r}, G_n)=\sum_{v \in V(G_n)} {d_v \choose r}$, where $d_v$ is the degree of the vertex $v \in V(G_n)$. 

%The $s$-star is the complete bipartite graph $K_{1, s}$. For $v\in [n]$, denote by $d_v$ the degree of the vertex $v$ in $G_n$. Then $N(K_{1,s}, G_n):=\sum_{v=1}^n {d_v \choose s}$ is the number of $s$-stars in $G_n$. %Also let $(d_1, d_2,\cdots d_n)$ denote the  degree sequence of $X$, i.e. $d_i:=\sum_{j\ne i}X_{ij}$.
 
Now, suppose $G_n$ is colored with $c_n$ colors as in \eqref{eq:uniform}. If $X_v$ denotes the color of vertex $v\in V(G_n)$, then the number of monochromatic copies of $K_{1, r}$ in $G_n$ is 
\begin{equation}
T(K_{1, r}, G_n):=\sum_{v=1}^{|V(G_n)|}\sum_{\bm u \in {V(G_n) \choose r}} a_v(\bm u, G_n) \bm 1\{X_v=X_{\bm u}\},
\label{eq:m2stars}
\end{equation} 
where 
\begin{itemize}
\item[--] ${V(G_n) \choose r}$ is the collection of $r$-element subsets of $G_n$;
\item[--] $a_v(\bm u, G_n)=\prod_{s=1}^r a_{vu_s}(G_n)$, for $v \in V(G_n)$ and $\bm u =\{u_1, u_2, \ldots, u_r\} \in {V(G_n) \choose r}$; 
\item[--] $\bm 1\{X_v=X_{\bm u}\}:=\bm 1\{X_v=X_{u_1}=\cdots=X_{u_r}\}$, for $v \in V(G_n)$ and $\bm u  \in {V(G_n) \choose r}$, as above. 
\end{itemize}
Note that $$\E(T(K_{1, r}, G_n))=\frac{1}{c_n^r}\sum_{v=1}^{|V(G_n)|}\sum_{\bm u \in {V(G_n) \choose r}} a_v(\bm u, G_n)=\frac{1}{c_n^r}N(K_{1, r}, G_n).$$ 
It is known that the limiting behavior of $T(K_{1, r}, G_n)$ is governed by its expectation: 

\begin{ppn}\label{ppn:bmm}\cite[Lemma 3.1]{BMM}
Let $\{G_n\}_{n\ge 1}$ be a sequence of deterministic graphs colored uniformly with $c_n$ colors as in \eqref{eq:uniform}. Then 
$$T(K_{1, r},G_n)\pto 
\left\{\begin{array}{ccc}
0  & \text{if}  &  \lim_{n\rightarrow\infty} \E(T(K_{1, r}, G_n))=0,\\
\infty  & \text{if}  & \lim_{n\rightarrow\infty} \E(T(K_{1, r}, G_n))=\infty.
\end{array}\right.
$$
\end{ppn} 

Therefore, the most interesting regime is where $\E(T(K_{1, r}, G_n))=\Theta(1)$,\footnote{For two non-negative sequences $(a_n)_{n \geq 1}$ and $(b_n)_{n \geq 1}$, $a_n=\Theta(b_n)$ means that there exist positive constants $C_1, C_2 $, such that $C_1 b_n \leq a_n \leq C_2 b_n$, for all $n$ large enough.} that is, $c_n\rightarrow \infty$ such that 
\begin{align}\label{eq:2starmean}
\E(T(K_{1, r}, G_n))=\frac{N(K_{1, r}, G_n)}{c_n^r}=\frac{1}{c_n^r}\sum_{v \in V(G_n)}{d_v \choose r}=\Theta(1). 
\end{align}

%The following theorem completely characterizes the limiting distribution of the number of monochromatic $K_{1, r}$, for any sequence of graphs satisfying \eqref{eq:2starmean}. 

\begin{thm}\label{th:2star}Let $\{G_n\}_{n \geq 1}$ be a sequence of graphs colored uniformly with $c_n$ colors, as in \eqref{eq:uniform}. Assume $c_n \rightarrow \infty$ such that the following hold: 

\begin{itemize} 
\item[(1)] For every $k \in [1, r+1]$, there exists $\lambda_k \geq 0$ such that
\begin{equation}\label{eq:Finduced}
	\lim_{n \rightarrow \infty}  \frac{\sum_{F\in \sC_{r,k}} N_{\mathrm{ind}}(F,G_n)}{c_n^{r}} = \lambda_k ,
\end{equation}
where $N_{\mathrm{ind}}(F, G_n)$ is the number of induced copies of $F$ in $G_n$ and $\sC_{r,k} := \{F \supseteq K_{1, r}: |V(F)|=r+1 \text{ and } N(K_{1, r},F)=k\}$.

\item[(2)] Let $d_{(1)}\geq d_{(2)} \geq \ldots \geq d_{(|V(G_n)|)}$ be the degrees of the vertices in $G_n$ arranged in non-increasing order, such that 
\begin{align}\label{eq:degassumption}
\lim_{n\rightarrow \infty}\frac{d_{(v)}}{c_n}=\theta_v,
\end{align}
for each $v \in V(G_n)$ fixed. 
%\begin{align}\label{eq:2trianglemean}
%\lim_{n \rightarrow \infty}\frac{N(K_3, G_n)}{c_n^2}=\nu.
%\end{align} 
\end{itemize}

Then 
\begin{align}\label{eq:2star}
T(K_{1, r}, G_n) \rightarrow \sum_{v=1}^\infty {T_v \choose r}+\sum_{k=1}^{r+1} k Z_k,
\end{align}
where  the convergence is in distribution and in all moments, and 
\begin{itemize}
\item[--] $T_1, T_2, \ldots, $ are independent  $\dPois(\theta_1), \dPois(\theta_2), \ldots $, respectively;
\item[--] $Z_1, Z_2, \ldots, Z_{r+1}$ are independent  $\dPois(\lambda_1-\frac{1}{r!}\sum_{u=1}^\infty \theta_u^r), \dPois(\lambda_2), \ldots \dPois(\lambda_{r+1})$, respectively;
\item[--] the collections $\{T_k, k \geq 1\}$ and $\{Z_k, 1 \leq k \leq r+1\}$ are independent. 
\end{itemize}
Conversely, if $T(K_{1, r}, G_n)$ converges in distribution, then the limit is necessarily of the form 
as in the RHS of \eqref{eq:2star}, for some non-negative constants $\theta_1 \geq \theta_2\ge \cdots$, and $\{\lambda_k, 1 \leq k \leq r+1 \}$.
%such that \eqref{eq:2starmeanN} and \eqref{eq:degassumption} hold along a subsequence.}
\end{thm}

This result gives a complete characterization of the limiting distribution of $T(K_{1, r}, G_n)$, in the regime where $\E(T(K_{1, r}, G_n))=\Theta(1)$ (in fact, under the assumptions of the theorem $\E(T(K_{1, r}, G_n))\rightarrow \sum_{k=1}^{r+1} k \lambda_k$). Note that the  limit in \eqref{eq:2star} has two components: 
\begin{itemize}

\item[--] a {\it non-linear part}   $\sum_{v=1}^\infty {T_v \choose r}$ which corresponds to the number of monochromatic $K_{1, r}$ in $G_n$ with central vertex of ``high" degree, that is, the vertices of degree $\Theta(c_n)$; and  

\item[--] a {\it linear part} $\sum_{k=1}^{r+1} k Z_k$ which is the number of monochromatic $K_{1, r}$ from the ``low"  degree vertices, that is, degree $o(c_n)$;
\end{itemize}
and, perhaps interestingly, the linear and the non-linear parts are asymptotically independent. The proof is given in Section \ref{sec:2starpf}. It involves decomposing the graph based on the degree of the vertices, and then using moment comparisons, to establish independence and compute the limiting distribution.

\begin{remark} An easy sufficient condition for \eqref{eq:Finduced} is the convergence of $\frac{1}{c_n^{|V(F)|-1}} N_{\mathrm{ind}}(K_{1, r}, G_n)$ for {\it every} super-graph $F$ of $K_{1, r}$ with $|V(F)|=r+1$. However, condition \eqref{eq:Finduced} does not require the convergence for every such graph, and is applicable to more general examples, as described below: Define a sequence of graphs $G_n$ as follows:
\[   
G_n = 
\begin{cases}
\text{ disjoint union of } n \text{ isomorphic copies of the 3-star } K_{1, 3} &\quad\text{ if } ~n \text{ is odd}\\
\text{ disjoint union of } n \text{ isomorphic copies of the (3, 1)-tadpole }  \Delta_+ &\quad\text{ if } ~n \text{ is even,}\\ 
\end{cases}
\]
where the (3, 1)-tadpole is the graph obtained by joining a triangle and a single vertex with a bridge. Now, choosing $c_n = \lfloor n^{1/3} \rfloor$, gives $\E(T(K_{1, 3}, G_n))\rightarrow 1$. In this case, 
$$ \frac{\sum_{F\in \sC_{H,1}} N_{\mathrm{ind}}(F,G_n)}{c_n^3} = \frac{N_{\mathrm{ind}}(K_{1, 3},G_n)+N_{\mathrm{ind}}(\Delta_+, G_n)}{c_n^3} \rightarrow 1,$$ 
and $\frac{1}{c_n^3} \sum_{F\in \sC_{H,4}} N_{\mathrm{ind}}(F,G_n)=\frac{1}{c_n^3} \sum_{F\in \sC_{H,4}} N_{\mathrm{ind}}(F,G_n)=\frac{1}{c_n^3} \sum_{F\in \sC_{H,2}} N_{\mathrm{ind}}(F,G_n)=0$. Therefore, Theorem \ref{th:2star} implies that $T(K_{1, 3}, G_n) \dto \dPois(1)$ (which can also be directly verified, because, in this case, $T(K_{1, 3}, G_n)$ is a sum of independent $\dBer(\frac{1}{c_n^3})$ variables).  However, it is easy to see that individually both $\frac{1}{c_n^3} N_{\mathrm{ind}}(K_{1, 3}, G_n)$ and $\frac{1}{c_n^3} N_{\mathrm{ind}}(\Delta_+, G_n)$ are non-convergent.  \end{remark}

The limit in \eqref{eq:2star} simplifies when the graph $G_n$ has no vertices of high degree. The following corollary is a  consequence of Theorem \ref{th:2star}.

\begin{cor}\label{cor:2star}Let $\{G_n\}_{n \geq 1}$ be a sequence of deterministic graphs. Then the following are equivalent.
\begin{enumerate}
\item[(a)] Condition \eqref{eq:Finduced} and $\lim_{n\rightarrow\infty}\frac{\Delta(G_n)}{c_n}=0$, where $\Delta(G_n):=\max_{v\in V(G_n)} d_v$. 

\item[(b)] $T(K_{1, r}, G_n) \dto \sum_{k=1}^{r+1} k Z_k$, where $Z_1, \ldots, Z_{r+1}$ are independent  $\dPois(\lambda_1), \ldots \dPois(\lambda_{r+1})$, respectively.  
\end{enumerate}
\end{cor}

The proof of the corollary is given in Section \ref{sec:cor2starpf}. Applications of this corollary and Theorem \ref{th:2star} are discussed in Section \ref{sec:2starexample}. In Section \ref{sec:conclusion} we discuss open problems and directions for future research.

\section{Proofs of Theorem \ref{th:2star} and Corollary \ref{cor:2star}}
\label{sec:2starpf}

The proof of Theorem \ref{th:2star} has four main steps: 

\begin{enumerate}
\item[(1)] Decomposing $G_n$ into the ``high"-degree and ``low"-degree vertices, and showing that the resulting error term vanishes (Section \ref{sec:pfoutline}).

\item[(2)] Showing that the contributions from the ``high"-degree and ``low"-degree vertices are asymptotically independent in moments (Section \ref{sec:pfm1}). 

\item[(3)] Computing the limiting distribution of the number of monochromatic $r$-stars with central vertex at one of the ``high"-degree vertices, which gives the non-linear term in \eqref{eq:2star} (Section \ref{sec:pfm2}).

\item[(4)] Computing the limiting distribution of the number of monochromatic $r$-stars from the ``low"-degree vertices, which gives the linear combination of independent Poisson variables in \eqref{eq:2star} (Section \ref{sec:pfm3}).
\end{enumerate} 
The proof of Theorem \ref{th:2star} can be easily completed by combining the above steps (Section \ref{sec:pf2star}). The proof of Corollary \ref{cor:2star} is given in Section \ref{sec:cor2starpf}.

Before proceeding we recall some standard asymptotic notation. For two nonnegative sequences $(a_n)_{n\geq 1}$ and $(b_n)_{n\geq 1}$, $a_n \lesssim b_n$ means $a_n=O(b_n)$, and $a_n \sim b_n$ means $a_n = (1+o(1))b_n$. We will use subscripts in the above notation, for example, $O_\square(\cdot)$, $\lesssim_\square$ to denote that the hidden constants may depend on the subscripted parameters.

\subsection{Decomposing $G_n$} 
\label{sec:pfoutline}

To begin with, note that the number of $r$-stars in $G_n$ remains unchanged if all edges $(u,v)$ in $G_n$ such that $\max\{d_u,d_v\} \le r-1$  are dropped.  Hence, without loss of generality,  assume that $\max\{d_u,d_v\}\ge r$, for all edges $(u,v)\in G_n$. This ensures that $N(K_{1, r}, G_n)=\sum_{v\in V(G_n)} {d_v \choose r}$ has the same order as $\sum_{v\in V(G_n)} d_v^r$ as shown below:

\begin{obs}\label{obs:edges} If $\max\{d_u,d_v\}\ge r$, for all edges $(u,v)\in G_n$, then assumption \eqref{eq:2starmean} implies \begin{align}\label{eq:2starmean_k}
\sum_{v\in V(G_n)}d_v^r=\Theta(c_n^r).
\end{align}
\end{obs}

\begin{proof} In this case, the following inequality holds 
\begin{align}\label{eq:edges}
\frac{1}{2} \sum_{v\in V(G_n)}d_v \le   \sum_{v\in V(G_n)}d_v \bm 1\{d_v\ge r \}.
\end{align}
To see this note that if an edge $(u,v) \in E(G_n)$ has $\min\{d_u,d_v\}\ge r$, then that edge is counted two times in the RHS above, and an edge $(u,v) \in E(G_n)$ which has $\min\{d_u,d_v\}\le r-1$ (but $\max\{d_u,d_v\}\ge r$)  is counted once in the RHS, whereas every edge of $E(G_n)$ is counted twice in the LHS. 

Then  
\begin{align*}
\sum_{v\in V(G_n)}d_v^r=&\sum_{v\in V(G_n)}d_v^r 1\{d_v<r\}+\sum_{v\in V(G_n)}d_v^r\{d_v\ge r\}\\
\le &(r-1)^{r-1}\sum_{v\in V(G_n)}d_v+r^r \sum_{v\in V(G_n)}{d_v\choose r}\\
\le &2r^{r-1}\sum_{v\in V(G_n)}d_v \bm 1\{d_v\ge r \}+r^r\sum_{v\in V(G_n)}{d_v\choose r} \tag*{(using \eqref{eq:edges})}\\
\le &2 r^r \sum_{v\in V(G_n)}{d_v\choose r}+ r^r \sum_{v\in V(G_n)}{d_v\choose r}=3 r^r \sum_{v\in V(G_n)}{d_v\choose r},
\end{align*}
from which the desired conclusion follows on using \eqref{eq:2starmean}. 
\end{proof}

Throughout the rest of this section, we will thus assume, that $\max\{d_u,d_v\}\ge r$, for all edges $(u,v)\in G_n$ and, hence, \eqref{eq:2starmean} implies $\eqref{eq:2starmean_k}$.
Note that \eqref{eq:2starmean_k} implies $$\Delta(G_n):=\max_{v\in V(G_n)}d_v =\Theta(c_n).$$
In fact, using \eqref{eq:2starmean_k} it can be shown that there are not too many vertices $v \in V(G_n)$ with $d_v=O(c_n)$. To this end, we have the following definition:

\begin{defn}\label{def:big} Fix $\varepsilon >0$, such that $\varepsilon\ne \theta_u$ for any $u\in \N$.  (This can be done, as the set $\{\theta_u, u\in \N\}$ is countable.) A vertex $v\in V(G_n)$ is said to be $\varepsilon$-{\it big} if $d_v\geq \varepsilon c_n$. Denote the subset of $\varepsilon$-big vertices by $V_{\varepsilon}(G_n)$. 
\end{defn}

The following lemma is an easy consequence of  \eqref{eq:2starmean_k} and the above definition. 

\begin{lem}\label{lm:deglarge} Assume \eqref{eq:2starmean_k} holds. Then for $n$ large enough,  number of $\varepsilon$-{\it big} vertices  $|V_\varepsilon(G_n)|$ does not depend on $n$. 
%$|V_\varepsilon(G_n)| \lesssim \frac{1}{\varepsilon^r}$. 
\end{lem}

\begin{proof}  Let $\eta=\eta(\varepsilon)\in \N$ be such that $\theta_\eta>\varepsilon>\theta_{\eta+1}$. (Note that such a $\eta$ exists for $\varepsilon$ small enough, whenever $\eta_0=\lim_{\varepsilon\rightarrow 0}\eta(\varepsilon)\ge 1$.\footnote{Since $\eta(\varepsilon)$ is monotonic non-increasing in $\varepsilon$, the limit $\eta_0:=\lim_{\varepsilon\rightarrow 0}\eta(\varepsilon)$ exists. If $\eta_0=0$, then $\max_{v\in V(G_n)}d_v=o(c_n)$, and the first term in the RHS of \eqref{eq:TS2_I} is trivially zero.}) Then for all $n$ large enough,  $d_{\eta+1}<\varepsilon c_n<d_\eta$, and so $|V_\varepsilon(G_n)|=\eta$. Thus the number of $\varepsilon$-{\it big} vertices is free of $n$, and depends only on $\varepsilon$.   
%
%Next, by Observation \eqref{obs:edges}, $$c_n^r \gtrsim \sum_{v\in V(G_n)}d_v^r\geq \sum_{v: d_v\geq \varepsilon c_n} d_v^r \geq |\{v: d_v\geq \varepsilon c_n\}|\varepsilon^r c_n^r,$$ 
%completing the proof of the lemma. 
\end{proof}

%Let $G_n[V_\varepsilon(G_n)]$ be the subgraph of $G_n$ induced by the $\varepsilon$-big vertices. 
Define $G_{n, \varepsilon}$ to be the subgraph of $G_n$  obtained by removing the edges between the $\varepsilon$-big vertices.  Denote by $T(K_{1, r}, G_{n, \varepsilon})$ the  number of monochromatic $r$-stars in $G_{n, \varepsilon}$. The following lemma shows that removing the edges between the $\varepsilon$-big vertices of $G_n$ does not change the number of monochromatic $r$-stars in $G_n$, in the limit:

\begin{lem}\label{lm:Geps} Assume \eqref{eq:2starmean} holds. Then for every fixed $\varepsilon>0$, as $n \rightarrow \infty$, $$N(F, G_{n})-N(F, G_{n, \varepsilon})=o(c_n^r) \quad  \text{and} \quad N_{\mathrm{ind}}(F, G_{n})-N_{\mathrm{ind}}(F, G_{n, \varepsilon})=o(c_n^r),$$
for all $F \supseteq K_{1, r}$  with $|V(F)|=r+1$. Consequently, 
$$\lim_{n \rightarrow \infty}\E|T(K_{1, r}, G_n)-T(K_{1, r}, G_{n, \varepsilon})|\rightarrow 0.$$
\end{lem}

\begin{proof} If a graph $F \supseteq K_{1, r}$ with $|V(F)|=r+1$ is a subgraph of $G_n$, but not a subgraph of $G_{n, \varepsilon}$, then it must have at least one edge with both end-points in $V_{\varepsilon}(G_n)$. Choosing this edge in $|V_\varepsilon(G_n)|^2$ ways and the remaining $r-1$ vertices in $O(c_n^{r-1})$ ways (since the maximum degree $\Delta(G_n) =\Theta(c_n)$), it follows that 
$$N(F, G_{n})-N(F, G_{n, \varepsilon})=O(c_n^{r-1}|V_\varepsilon(G_n)|^2)=o(c_n^r),$$ as $n \rightarrow \infty$, since by Lemma \ref{lm:deglarge} $|V_\varepsilon(G_n)|=O_\varepsilon(1)$. As the number of induced copies of $F$ in $G_n$ which are not in $G_{n, \varepsilon}$, is bounded by the total number of copies of $F$  in $G_{n, \varepsilon}$ which are not in $G_n$, the result on induced copies follows.

In particular,  
\begin{align*}%\label{2stardifference}
\E|T(K_{1, r}, G_n)-T(K_{1, r}, G_{n, \varepsilon})|\lesssim \frac{1}{c_n^r}(c_n^{r-1} |V_\varepsilon(G_n)|^2)=\frac{1}{c_n}|V_\varepsilon(G_n)|^2 \rightarrow 0,
\end{align*}
as $c_n\rightarrow \infty$. 
\end{proof}

We now decompose the graph $G_{n, \varepsilon}$ based on the degree of the vertices as follows: 
\begin{itemize}
\item[--]
Let $G_{n, \varepsilon}^{+}$ be the sub-graph of $G_{n, \varepsilon}$ formed by the $\varepsilon$-big vertices and the edges incident on them. More formally, it has vertex set $V_\varepsilon(G_n) \bigcup N_{G_{n, \varepsilon}}(V_\varepsilon(G_n))$, where $N_{G_{n, \varepsilon}}(V_\varepsilon(G_n))$ is neighborhood of $V_\varepsilon(G_n)$ in $G_{n, \varepsilon}$,\footnote{For a graph $H=(V(H), E(H))$ and $S\subseteq V(H)$, the  {\it neighborhood} of $S$ in $H$ is $N_H(S)=\{v \in V(H): \exists ~u \in S \text{ such that } (u, v) \in E(H)\}$.} and edge set $\{(u, v) \in G_{n, \varepsilon}: v\in V_\varepsilon(G_n)\}$. Note that by construction $G_{n,\varepsilon}^+$ is a bipartite graph.

\item[--] Let $G_{n, \varepsilon}^{-}$ denote the induced subgraph of $G_{n, \varepsilon}$ with vertex set $V(G_n)\setminus V_\varepsilon(G_n)$.
\end{itemize}

The decomposition of the graph $G_{n, \varepsilon}$ is illustrated in Figure \ref{fig:G12}. Note that $G_{n, \varepsilon}^{+}$ and $G_{n, \varepsilon}^{-}$ have common vertices (the black vertices in Figure \ref{fig:G12}), but no common edges, and consequently no common $r$-stars. This implies 
$$T(K_{1, r}, G_{n, \varepsilon})= T_+(K_{1, r}, G_{n, \varepsilon}^{+})+T(K_{1, r}, G_{n, \varepsilon}^{-})+R(K_{1, r}, G_{n, \varepsilon}),$$
where $T(K_{1, r}, G_{n, \varepsilon}^{-})$ is the number of monochromatic $r$-stars in $G_{n, \varepsilon}^{-}$; and (recalling the definition of $a_v(\bm u, G_n)$ from \eqref{eq:m2stars})
\begin{align}\label{eq:T1}
T_+(K_{1, r}, G_{n, \varepsilon}^{+}):=\sum_{v=1}^{|V_\varepsilon(G_n)|}\sum_{\bm u \in {V(G_n) \choose r}} a_{vu_1}(G_n) a_v(\bm u, G_n) \bm 1\{X_v=X_{u_1}=X_{\bm u}\},  
%\sum_{v \in V_\varepsilon(G_n)}\sum_{1\le u<w\le |V(G_n)|}a_{vu}(G_{n, \varepsilon})a_{vw}(G_{n, \varepsilon})\bm 1\{X_v=X_{u}=X_{w}\}, 
\end{align} 
counts the number of monochromatic $r$-stars in $G_{n, \varepsilon}^+$ with central vertex in $V_\varepsilon(G_n)$;\footnote{Note that $T_+(K_{1, r}, G_{n, \varepsilon}^{+})$ is not the number of $r$-stars in $G_{n, \varepsilon}^+$: It does not include the $r$-stars in $G_{n, \varepsilon}^+$ with central vertex in $N_{G_{n,\varepsilon}}(V_\varepsilon(G_n))$ (the black vertices in Figure \ref{fig:G12}). Instead, these $r$-stars are included in the remainder term $R(K_{1, r}, G_{n, \varepsilon})$.} and the {\it remainder term}  
\begin{align}\label{eq:delta}
R(K_{1, r}, G_{n, \varepsilon})&:=\sum_{v \notin V_\varepsilon(G_n)}\sum_{u_1\in V_\varepsilon(G_n)} \sum_{\bm u \in {V(G_n) \choose r-1}}  a_{v u_1}(G_n)a_v(\bm u, G_{n}) \bm 1\{X_v=X_{\bm u}\}.
%\sum_{u\in V_\varepsilon(G_n)}\sum_{w \in V(G_n)}a_{vu}(G_{n, \varepsilon})a_{vw}(G_{n, \varepsilon})\bm 1\{X_v=X_{u}=X_{w}\}.
\end{align}

%%%%%%%%%%%%%%%%%%%%%%%%%%%%%%%%%%%%%%%%%%%%%%%%%%%%%%%%%%%%%%%%%%%%%%%%%%%%%%%%
\begin{figure*}[h]
\centering
\begin{minipage}[c]{1.0\textwidth}
\centering
\includegraphics[width=3.0in]
    {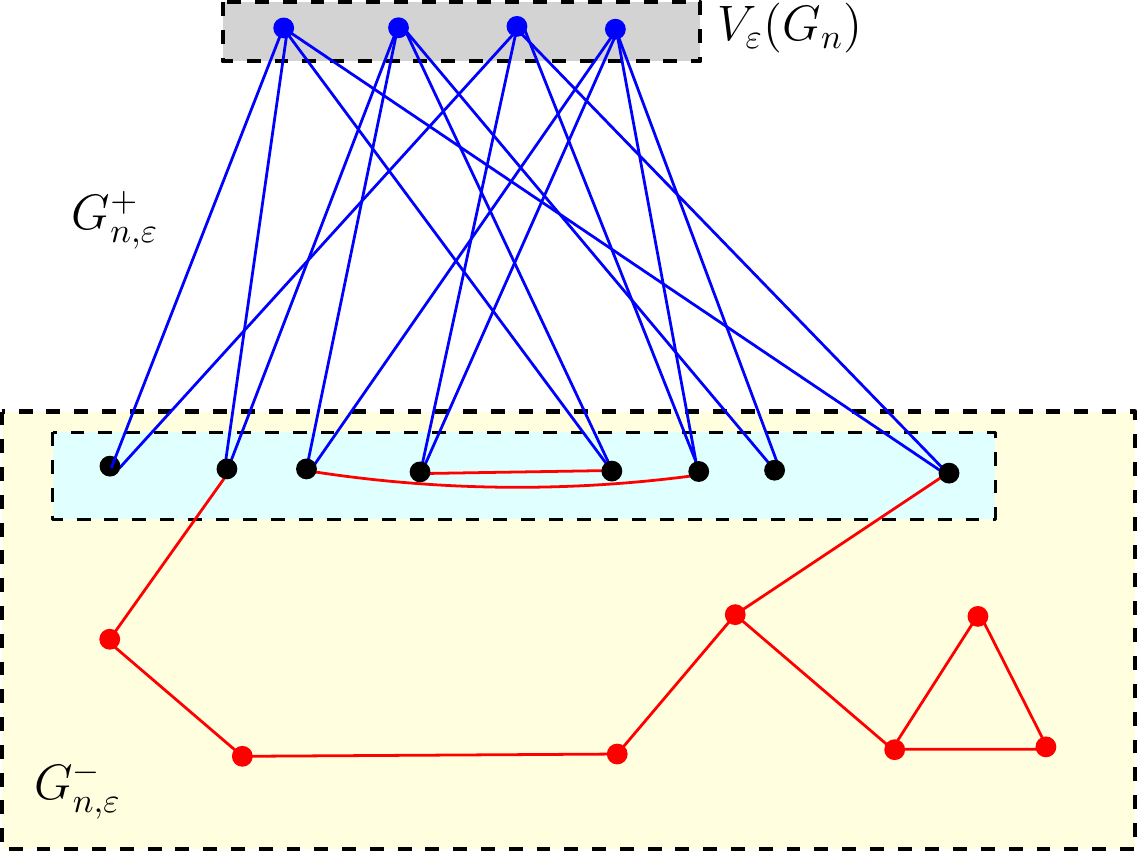}\\
%\small{(a)}
\end{minipage}
\caption{\small{The decomposition of $G_{n, \varepsilon}$: The graph formed by the blue edges is $G_{n, \varepsilon}^{+}$ and the graph formed by the red edges is $G_{n, \varepsilon}^{-}$. Note that the black vertices belong to both  $G_{n, \varepsilon}^{+}$ and $G_{n, \varepsilon}^{-}$.}}
%\vspace{-0.15in}
\label{fig:G12}
\end{figure*}
%%%%%%%%%%%%%%%%%%%%%%%%%%%%%%%%%%%%%%%%%%%%%%%%%%%%%%%%%%%%%%%%%%%%%%%%%%%%%%%%%%%%%%%%%%%%%%%%%%%%%%%%%%

The following lemma shows that the remainder term goes to zero in expectation, and therefore, in probability. 
  
\begin{lem}\label{lm:TS2} Let $R(K_{1, r}, G_{n, \varepsilon})$ be as defined above in \eqref{eq:delta}. Under the assumptions of Theorem \ref{th:2star}, 
$$\lim_{\varepsilon\rightarrow 0}\lim_{n \rightarrow \infty}\E(R(K_{1, r}, G_{n, \varepsilon}))=0.$$ 
\end{lem}

\begin{proof}Note that for $v \notin V_\varepsilon(G_n)$ $$ \sum_{\bm u \in {V(G_n) \choose r-1}}  a_v(\bm u, G_{n}) \le d_v^{r-1}\leq (\varepsilon c_n)^{r-1}.$$ 
Moreover $\sum_{v \notin V_\varepsilon(G_n)} a_{vu_1}(G_{n, \varepsilon})  \leq d_{u_1}$. Then, using \eqref{eq:delta}, for any $M>0$,
\begin{align}\label{eq:TS2_I}
\E(R(K_{1, r}, G_{n, \varepsilon}))\leq &  \frac{(\varepsilon c_n)^{r-1}}{c_n^r}  \sum_{u_1 \in V_\varepsilon(G_n)} d_{u_1} \nonumber \\
=& \frac{\varepsilon^{r-1}}{c_n} \left(\sum_{u:  \varepsilon c_n\leq d_u < M \varepsilon c_n} d_u+ \sum_{u: d_u\ge M\varepsilon c_n} d_u \right) \nonumber \\
\le &\frac{1}{c_n^r} \sum_{u: \varepsilon c_n\le d_u  < M \varepsilon c_n} d_u^r+ \frac{1}{M^{r-1}c_n^r} \sum_{u\in V(G_n)} d_u^r. 
\end{align} 
Since $\limsup_{n\rightarrow \infty}\frac{1}{c_n^r}\sum_{u\in V(G_n)}d_u^r<\infty$ (from Observation \ref{obs:edges}), the second term in the RHS of \eqref{eq:TS2_I} converges to $0$ on letting $n\rightarrow\infty$ followed by $M\rightarrow\infty$. 

Next, recall that  $\eta=\eta(\varepsilon)$ is such that $\theta_{\eta+1}< \varepsilon< \theta_\eta$.  
%Also by choice of $\varepsilon$ we have $\theta_u \ne M\varepsilon$ for any $u$, and so fixing $M$  there exists $q$ be such that $\theta_{q-1}>M\varepsilon> \theta_{K}$. Note that $q$ de\etaend on $M,\varepsilon$ but is free of $n$. Also for all $n$ large we have $d_{q-1}>M \varepsilon c_n>d_K$ and $d_\eta>\varepsilon c_n>d_{\eta+1}$. 
Thus, for all $n$ large enough, $d_{\eta+1}<\varepsilon c_n<d_\eta$, and as $n \rightarrow \infty$, the first term in the RHS above becomes 
\begin{align*}
\limsup_{n\rightarrow\infty}\sum_{u=1}^\eta \frac{d_u^r}{c_n^r} \bm 1\{d_u<M\varepsilon c_n\}\le \sum_{u=1}^\eta\theta_u^r \bm 1\{\theta_u\le M\varepsilon\}=\sum_{u=1}^\infty  \theta_u^r \bm 1\{\varepsilon\le \theta_u \leq M \varepsilon \},
\end{align*} which converges to $0$ on letting $\varepsilon\rightarrow 0$, by using DCT along with the fact that $\sum_{u=1}^\infty \theta_u^r\leq \limsup_{n\rightarrow \infty}\frac{1}{c_n^r}\sum_{u\in V(G_n)}d_u^r<\infty$ (by Fatou's lemma). 
\end{proof}

%If $p_0=0$, then  $p(\varepsilon)=0$, and $V_\varepsilon(G_n)$ is empty, for all $\varepsilon>0$. This is the case which corresponds to $\max_{v\in V(G_n)}d_v=o(c_n)$. The case for $p_0=0$ is simpler, and will be dealt with as part of the proof of Theorem \ref{th:2star}. Hence, throughout the rest of this section we consider the more interesting case $p_0\ge 1$. This also includes the case $p_0=\infty$.

Combining Lemma \ref{lm:Geps} and Lemma \ref{lm:TS2} it follows that 
\begin{align}\label{eq:m2starreduction}
T(K_{1, r}, G_n)=&T(K_{1, r}, G_{n, \varepsilon})+o_P(1) \nonumber\\
= &T_+(K_{1, r}, G_{n, \varepsilon}^{+})+T(K_{1, r}, G_{n, \varepsilon}^{-})+o_P(1).
\end{align}
Therefore, the limiting distribution of the $T(K_{1, r}, G_n)$ is the same as that of $T_+(K_{1, r}, G_{n, \varepsilon}^{+})+T(K_{1, r}, G_{n, \varepsilon}^{-})$. 
%The next Lemma studies the joint distribution of the vector $(M_v,v\in V_\varepsilon(G_n);T(K_{1, r}, G_{n, \varepsilon}^{-}))$.

\subsection{Independence in Moments of the Contributions from $G_{n, \varepsilon}^+$ and $G_{n, \varepsilon}^-$}
\label{sec:pfm1}

In this section we show that the number of monochromatic $K_{1, r}$ coming from  $G_{n, \varepsilon}^+$ and $G_{n, \varepsilon}^-$ are asymptotically independent in moments. Without loss of generality, assume the vertices in $V(G_n)$ are labelled $1, 2, \ldots, |V(G_n)|$ such that $d_1\geq d_2 \geq \cdots \geq d_{|V(G_n)|}$, and $\eta=\eta(\varepsilon)$ such that $\theta_{\eta+1}< \varepsilon< \theta_\eta$ (assuming $\eta_0=\lim_{\varepsilon\rightarrow 0}\eta(\varepsilon)\ge 1$). Then, by definition \eqref{eq:T1}, 
\begin{align}\label{eq:T1eps}
T_+(K_{1, r}, G_{n, \varepsilon}^{+})=\sum_{v=1}^\eta {T_{G_{n, \varepsilon}^+}(v)\choose r},\quad \text{where }  T_{G_{n, \varepsilon}^+}(v):=\sum_{u\in V(G_{n, \varepsilon})} a_{uv}(G_n)1\{X_u=X_v\},
\end{align}
is the number of monochromatic $r$-stars in $G_{n, \varepsilon}$, with central vertex $v \in V_{\varepsilon}(G_n)$.

Now, fix a finite positive integer $K \le \eta_0$. Then, for $\varepsilon >0$ small enough, $\eta(\varepsilon)\ge K$, and so $\{T_{G_{n, \varepsilon}^+}(v): 1 \leq v \leq K\}$ are well defined. The following lemma shows that this collection and $T(K_{1, 2}, G_{n, \varepsilon}^{-})$ are asymptotically independent in the moments. % the joint moments of  $T_+(K_{1, 2}, G_{n, \varepsilon}^{+})$ and $T(K_{1, 2}, G_{n, \varepsilon}^{-})$ factorize in the limit

\begin{lem}\label{lm:ind} Assume \eqref{eq:2starmean} holds. Then for every finite $K \le \eta_0$ and non-negative integers $s, t_1,\cdots,t_K$, 
\begin{align}\label{eq:moment_ind}
\lim_{\varepsilon \rightarrow 0} \lim_{n\rightarrow\infty}\left|\E\left( {T}(K_{1, 2}, G_{n, \varepsilon}^{-})^s  \prod_{v=1}^K T_{G_{n, \varepsilon}^+}(v)^{t_v} \right)-\E  {T}(K_{1, 2}, G_{n, \varepsilon}^{-})^s \left(\E\prod_{v=1}^{K} T_{G_{n, \varepsilon}^+}(v)^{t_v}\right) \right|=0.
\end{align}
\end{lem}

\subsubsection*{\textbf{Proof of Lemma} \ref{lm:ind}} For any labeled subgraph $H$ of $G_n$, define 
\begin{align}\label{eq:beta}
\beta(H):=\E\prod_{(u,v)\in E(H)} \bm 1\{X_u=X_v\}=\left(\frac{1}{c_n}\right)^{|V(H)|-\nu(H)},
\end{align}
where $\nu(H)$ is the number of connected components of $H$. Note that the definition of $\beta(\cdot)$ is invariant to the labelling of $H$, and so, it extends to unlabelled graphs as well. Thus, without loss of generality, we will define $\beta(H)$ as in \eqref{eq:beta}, for an unlabelled graph $H$ as well.

Let $H_1=(V(H_1), E(H_1))$ and $H_2=(V(H_2), E(H_2))$ be two (labelled) subgraphs of $G_n$, that is,  $V(H_1)$ and $V(H_2)$ are subsets of $V(G_n)$, which inherits the labelling induced by $V(G_n)$, and $E(H_1)$ and $E(H_2)$ are subsets of $E(H)$. Let $H_1\bigcup H_2=(V(H_1) \bigcup V(H_2), E(H_1)\bigcup E(H_2))$.

\begin{lem}\label{lm:beta_H} For any two finite graphs $H_1$ and $H_2$, $\beta\left(H_1\bigcup H_2\right)\ge \beta(H_1)\beta(H_2)$, where $\beta(\cdot)$ is defined above in \eqref{eq:beta}. 
\end{lem}

\begin{proof} Denote by $F=H_1\bigcup H_2$, and let $F_1, F_2, \ldots , F_{\nu(F)}$ be the connected components of $F$. Define
\begin{align}\label{eq:I12}
I_{1}=\{s \in [\nu(F)]:  V(F_s)\bigcap V(H_1)\ne \emptyset \text{ and } V(F_s)\bigcap V(H_2)= \emptyset \}, \nonumber \\
I_{2}=\{s \in [\nu(F)]:  V(F_s)\bigcap V(H_1)= \emptyset \text{ and } V(F_s)\bigcap V(H_2)\ne \emptyset \}, \nonumber \\
I_{12}=\{s \in [\nu(F)]:  V(F_s)\bigcap V(H_1)\ne \emptyset \text{ and } V(F_s)\bigcap V(H_2)\ne \emptyset \}. \end{align}
Fix $s \in I_{12}$, that is, $V(F_s)\bigcap V(H_1)\ne \emptyset$ and $V(F_s)\bigcap V(H_2)\ne \emptyset$. Then $F_s=F_{s}'\bigcup F_{s}''$, where $$F_{s}'=(V(F_s)\bigcap V(H_1), E(F_{s}) \bigcap E(H_1)), \quad \text{and} \quad F_{s}''=(V(F_s) \bigcap V(H_2), E(F_s)\bigcap E(H_2)).$$
Let $F_{s1}', F_{s2}' \ldots F_{sa}'$ be the connected components of $F_s'$ and similarly, $F_{s1}'', F_{s2}'' \ldots F_{sb}''$ be the connected components of $F_s''$, where $a=\nu(F_s')$ and $b=\nu(F_s'')$. Construct a bipartite graph $B_s=(B_s'\bigcup B_s'', E(B_s))$, where $B_s'=\{F_{s1}', F_{s2}', \ldots F_{sa}'\}$ and $B_s''=\{F_{s1}'', F_{s2}'', \ldots F_{sb}''\}$ and 
there is any edge between $F_{sx}'$ and $F_{sy}''$ if and only if $V(F_{sx}')\bigcap V(F_{sy}'') \ne \emptyset$, for $x \in [a]$ and $y \in [b]$. 
Note that $|V(F_s')\bigcap V(F_s'')| \geq |E(B_s)|$, and since the graph $F_s$ is connected, the graph $B_s$ is also connected. Therefore, 
$$|V(F_s')\bigcap V(F_s'')| \geq |E(B_s)|\geq |V(B_s)|-1 = \nu(F_s')+\nu(F_s'')-1,$$
This implies,  
\begin{align*}
|V(F_s)|=|V(F_s')|+|V(F_s'')|-|V(F_s')\bigcap V(F_s'')| \le |V(F_s')|-\nu(F_s')+|V(F_s'')|-\nu(F_s'')+1.
\end{align*}
Then, recalling \eqref{eq:I12}, it follows that 
\begin{align*}
\beta(H)=&\prod_{s\in I_1} \beta(F_s) \prod_{s\in I_2} \beta(F_s) \prod_{s \in I_{12}} \beta(F_s)\\
=&  \prod_{s \in I_{12}} \left(\frac{1}{c_n}\right)^{|V(F_s)|-1} \prod_{s\in I_1} \beta(F_s) \prod_{s\in I_2} \beta(F_s) \\
\ge & \left(\prod_{s \in I_{12}} \left(\frac{1}{c_n}\right)^{|V(F_s')|-\nu(F_s')} \prod_{s\in I_1} \beta(F_s) \right)\left( \prod_{s \in I_{12}} \left(\frac{1}{c_n}\right)^{|V(F_s'')|-\nu(F_s'')} \prod_{s\in I_2} \beta(F_s) \right) \\
=&\beta(H_1)\beta(H_2),
\end{align*}
completing the proof of the lemma.
\end{proof}

Now, recall the definitions of the graph $G_{n, \varepsilon}^{-}$ from Section \ref{sec:pfoutline}, %Now, by definition \eqref{eq:T1}, 
%\begin{align}\label{eq:G1}
%T_+(K_{1, 2}, G_{n, \varepsilon}^{+})&=\sum_{v \in V_\varepsilon(G_n)}\sum_{1\le u<w\le |V(G_n)|}a_{vu}(G_{n, \varepsilon})a_{vw}(G_{n, \varepsilon})\bm 1\{X_v=X_{u}=X_{w}\}\nonumber\\
%&=\sum_{(u, v, w) \in \sS(G_{n, \varepsilon}^{+})} \bm 1\{X_u=X_{v}=X_{w}\}, 
%\end{align}
%where $\sS(G_{n, \varepsilon}^{+})$ is the collection of  ordered triplets $(u, v, w)$, with $u, v, w \in V(G_{n, \varepsilon}^{+})$ such that $(u, v), (v, w) \in E(G_{n, \varepsilon}^{+})$, that is, $(u, v, w)$ is a 2-star in $G_{n, \varepsilon}^{+}$ with central vertex $v\in V(G_{n,\varepsilon})$.  
and note that
\begin{align}\label{eq:G2}
T(K_{1, r}, G_{n, \varepsilon}^{-})&= \sum_{\bm u \in \sS_r(G_{n, \varepsilon}^{-})} \bm 1\{X_{=\bm u}\}, 
\end{align}
where 
\begin{itemize}
\item[--] $\sS_r(G_{n, \varepsilon}^{-})$ is the collection of  ordered $(r+1)$-tuples $\bm u=(u_0, u_1,\cdots,u_r)$, such that $u_0, u_1,\ldots, u_r \in V(G_{n, \varepsilon}^{-})$ are distinct and $(u_0,u_i) \in E(G_{n, \varepsilon}^{-})$, for $i \in [1, r]$; and 

\item[--] $\bm 1\{X_{=\bm u}\}= \bm 1\{X_{u_0}=X_{u_1}=\cdots=X_{u_r}\}$. 
\end{itemize}

For any $u\in V(G_n)$, let $N_{G^{+}_{n,\varepsilon}}(u)$ be the neighborhood of $u$ in $G^{+}_{n,\varepsilon}$. Index the vertices in $N_{G^{+}_{n,\varepsilon}}(u)$ as $\{b_1(v), b_2(v), \ldots b_{d_v^{+}}(v)\}$, where $d_v^{+}$ is the degree of the vertex $v$ in $G^{+}_{n,\varepsilon}$. Let 
$$\Gamma= \prod_{v=1}^{K}N_{G^{+}_{n,\varepsilon}}(v)^{t_v}\times \sS_r(G_{n, \varepsilon}^{-})^s$$ 
denote the collection of vertices $\{b_j(v),1\le j\le t_v,1\le v\le K\}$ and $s$ ordered $(r+1)$-tuples $$\bm u_1=(u_{10},u_{11}, u_{12}, \cdots, u_{1r}), \bm u_2=(u_{20},u_{21}, u_{22}, \cdots, u_{2r}), \ldots, \bm u_s=(u_{s0},u_{s1}, \ldots  u_{sr}),$$ such that $b_j(v)\in N_{G^{+}_{n,\varepsilon}}(v)$, for $j\in [t_v]$ and $v \in [1,K]$, and $\bm u_a \in \sS_r(G_{n, \varepsilon}^{-})$, for $a \in [1, s]$. 

Then expanding the product ${T}(K_{1, 2}, G_{n, \varepsilon}^{-})^s  \prod_{v=1}^K T_{G_{n, \varepsilon}^+}(v)^{t_v}$ over the sum, the LHS of \eqref{eq:moment_ind} can be bounded above by: 
\begin{align}\label{eq:upper1}
& \sum_{\Gamma} \Bigg|\E \left(\prod_{v=1}^K \prod_{j=1}^{t_v} \bm 1\{X_{v}=X_{b_j(v)}\} \E \prod_{a=1}^{s}   \bm 1\{X_{=\bm u_{a}}\} \right) -\prod_{v=1}^K \left(\E\prod_{j=1}^{t_v} \bm 1\{X_{v}=X_{b_j(v)}\}\E \prod_{a=1}^{s}   \bm 1\{X_{=\bm u_{a}}\}\right) \Bigg| \nonumber\\
& = \sum_{\Gamma} \left|\beta\left(H_1\bigcup H_2\right)-\beta(H_1)\beta(H_2)\right|
\end{align}
where $\beta(\cdot)$ is defined in \eqref{eq:beta} and 
\begin{itemize}
\item[--] $H_1$ is the simple labelled subgraph of  $G_{n,\varepsilon}^+$ obtained by the union of the edges $(v, b_j(v))$ for $j \in [1, t_v]$ and $v\in [1,K]$.
\item[--] $H_2$ is the simple labelled subgraph of $G_{n, \varepsilon}^{-}$ obtained by the union of the $r$-stars formed by the collection of $(r+1)$-tuples $\{\bm u_{1},\cdots, \bm u_{s}\}$. More formally, 
$H_2=(V(H_2), E(H_2))$, where  
\begin{align*}%\label{eq:Sgraph}
V(H_2)=\bigcup_{j=1}^s \bm u_j \quad \text{and} \quad  E(H_2)=\bigcup_{j=1}^s \left\{(u_{j0}, u_{ja}): 1 \leq a \leq r\right\}.
\end{align*}

%\item[--] $H=H_1\bigcup H_2$.
\end{itemize}
Note that if $V(H_1)\bigcap V(H_2)=\emptyset$, then $\beta\left(H_1\bigcup H_2\right)=\beta(H_1)\beta(H_2)$, and so without loss of generality we may assume that the sum over $\Gamma$ includes only terms for which $H_1\bigcap H_2\ne \emptyset$.

\begin{defn}\label{defn:rs}
Let $\cH_{m_1,m_2}$ denote the set of all unlabelled graphs $H=(V(H), E(H))$ %is called $r$-{\it admissible} 
which can be formed by the union of $m_1$ edges and $m_2$ copies of $K_{1,r }$.%\footnote{We say an unlabeled graph $H$ is the union of 2 isomorphic copies of $K_{1, r}$ if $H$ can be obtained by identifying $t$ vertices of one copy of $K_{1, r}$, for some $t \in [r+1]$, with $t$ vertices of the other copy of $K_{1, r}$. Inductively, we say a graph $H$ is the union of $m \geq 3$ isomorphic copies of $K_{1, r}$, if $H$ can be obtained by identifying $t$ vertices of one copy of $K_{1, r}$, for some $t \in [r+1]$, with $t$ vertices of some graph $F$ which is the union of $m-1$  isomorphic copies of $K_{1, r}$.} 
\end{defn}

Now, recalling that $\beta\left(H_1\bigcup H_2\right)=\beta(H_1)\beta(H_2)$, if $V(H_1)\bigcap V(H_2)=\emptyset$, and $\beta\left(H_1\bigcup H_2\right)\ge \beta(H_1)\beta(H_2)$ otherwise, the RHS of \eqref{eq:upper1} can be bounded as follows:
\begin{align}\label{eq:upper2}
\sum_{\Gamma} \left|\beta\left(H_1\bigcup H_2\right) -\beta(H_1)\beta(H_2)\right|
& \le  \sum_{\Gamma}\beta\left(H_1 \bigcup H_2\right)  \nonumber\\
& =\sum_{m_1=1}^{s_1}\sum_{m_2=1}^{s_2} \sum_{H \in \mathcal H_{m_1,m_2}} \sum_{\Gamma:  H_1\bigcup H_2 \cong H} \beta\left(H_1\bigcup H_2\right)  \nonumber  \\
& \lesssim  \sum_{m_1=1}^{s_1}\sum_{m_2=1}^{s_2} \sum_{H \in \mathcal H_{m_1,m_2}} \beta(H) N(H, G_{n, \varepsilon}^+[K],G_{n,\varepsilon}^-),
\end{align}
where 
\begin{itemize}
\item[--] $G_{n, \varepsilon}^+[K]$ be the induced sub-graph of $G_{n, \varepsilon}^+$ formed by the vertices labeled $\{1, 2, \ldots, K\}$, that is, the $K$ highest degree vertices in $G_n$; and 
\item[--] $N(H, G_{n, \varepsilon}^+[K], G_{n,\varepsilon}^-)$ is the number of copies of $H=H_1\bigcup H_2$ in $G_{n, \varepsilon}^+[K]\bigcup G_{n,\varepsilon}^-$, such that $H_1$ is formed by the union of $m_1$ edges from $G_{n, \varepsilon}^+[K]$ and $H_2$ is formed by the union of $m_2$ copies of $K_{1,r }$ from $G_{n,\varepsilon}^-$, and $V(H_1)\bigcap V(H_2)\ne \emptyset$. 
\end{itemize}

Now, using $\beta(H)=\frac{1}{c_n^{|V(H)|-\nu(H)}}$ (by Lemma \ref{lm:beta_H}), and since the sum over $m_1,m_2, H$ in \eqref{eq:upper2} are all finite, to prove \eqref{eq:moment_ind} it suffices to show that for every $H\in \cH_{m_1,m_2}$, 
\begin{align}\label{eq:upperbd}
\limsup_{\varepsilon\rightarrow0}\limsup_{n\rightarrow\infty}\frac{N(H,G_{n, \varepsilon}^+[K],G_{n,\varepsilon}^-)}{c_n^{V(H)-\nu(H)}}=0.
\end{align}

To this end, fix  $H\in \cH_{m_1,m_2}$ such that $H=H_1\bigcup H_2$, such that $H_1$ is formed by the union of $m_1$ edges from $G_{n, \varepsilon}^+[K]$ and $H_2$ is formed by the union of $m_2$ copies of $K_{1,r }$ from $G_{n,\varepsilon}^-$, and $V(H_1)\bigcap V(H_2)\ne \emptyset$ (otherwise $N(H, G_{n, \varepsilon}^+[K], G_{n,\varepsilon}^-)=0$). Let $C_1, C_2, \ldots, C_{\nu(H)}$ the connected components of $H$. Fix $1 \leq j \leq \nu(H)$ and consider the following three cases:

\begin{itemize}
\item[--] $V(C_j)$ only intersects $V(H_1)$. Since $G_{n, \varepsilon}^+[K]$ is a bi-partite graph with bi-partition with $|E(G_{n,\varepsilon}^+[K])| \le K \Delta(G_n)  \lesssim_r K c_n$ (using $\Delta(G_n) = O(c_n)$). This gives
$$N(C_j, G_{n, \varepsilon}^+[K])\le |E(G_{n, \varepsilon}^+[K])|^{|V(C_j)|-1}\lesssim_{r, m_1} (K  c_n)^{|V(C_j)|-1}.$$

\item[--] $V(C_j)$ only intersects $V(H_2)$. Then there exists $1 \leq h \le m_2$ such that $H_2$ is spanned by $h$ isomorphic copies of $K_{1, r}$. Thus, using the bounds $N(K_{1, r},G_n)=\Theta(c_n^r)$ and gives the bound
$$N(C_j, G_{n,\varepsilon}^-)\le N(K_{1, r},G_n) \Delta(G_n)^{|V(C_j)|-r+1}\lesssim_{r, m_2} c_n^{|V(C_j)|-1},$$ 
using $\Delta(G_n) = O(c_n)$. 

\item[--] $V(C_j)$ intersects both $V(H_1)$ and $V(H_2)$.  If $C_j$ is such that it intersects both $H_1$ and $H_2$, then there is a vertex $v\in V(H_1)\bigcap V(H_2)$, such that $(u,v)$ is an edge in $G_{n, \varepsilon}^+[K]$, and $(v,w)$ is an edge $G_{n,\varepsilon}^-$. Thus, using the estimate  $\Delta(G_n)=O(c_n)$, 
\begin{align*} 
N(C_j,G_{n, \varepsilon}^+[K],G_{n,\varepsilon}^-)&\lesssim |E(G_{n, \varepsilon}^+[K])| \left(\max_{v\in V(G_{n,\varepsilon}^-)}d_v \right) \Delta(G_n)^{|V(C_j)|-3} \\ 
&\lesssim_{r, m_1, m_2} K\varepsilon c_n^{|V(C_j)|-1}.
\end{align*} 
\end{itemize}

Taking a product over $1 \leq j \leq \nu(H)$ and, since $V(H_1)\bigcap V(H_2)\ne \emptyset$, gives
\begin{align*}
N(H, G_{n, \varepsilon}^+[K],G_{n,\varepsilon}^-) \lesssim_{r, m_1, m_2}& \varepsilon K^{|V(H)|-\nu(H)}c_n^{|V(H)|-\nu(H)},
\end{align*}
which implies \eqref{eq:upperbd}, from which the desired conclusion follows. \qed \\

\subsection{Contribution from $G_{n, \varepsilon}^+$}
\label{sec:pfm2}

In this section we compute the asymptotic distribution of  $T_+(K_{1, r}, G_{n, \varepsilon}^{+})$ (recall \eqref{eq:T1}). This involves showing that the collection $\{T_{G_{n, \varepsilon}^+}(v): 1\le v\le K\}$ are asymptotically independent, by another moment comparison.

\begin{lem}\label{lm:ind2} Assume  \eqref{eq:2starmean} holds, and $\varepsilon>0$ small enough. Then for all non-negative integers $s_1,\cdots,s_K$,  
\begin{align}\label{eq:moment_ind}
\lim_{n\rightarrow\infty}\left|\E\left( \prod_{v=1}^K T_{G_{n, \varepsilon}^+}(v)^{s_v}\right)- \prod_{v=1}^{K} \E T_{G_{n, \varepsilon}^+}(v)^{s_v} \right|=0.
\end{align} 
As a consequence, $T_+(K_{1, r}, G_{n, \varepsilon}^{+}) \dto \sum_{v=1}^\eta {T_v \choose r}$, as $n \rightarrow \infty$, where $T_1, T_2, \ldots, T_\eta$ are independent $\dPois(\theta_1), \dPois(\theta_2), \ldots, \dPois(\theta_\eta)$, respectively. $($Recall that $\eta = \eta(\varepsilon)$ is such that $\theta_{\eta+1}< \varepsilon < \theta_\eta$.$)$
\end{lem}

\begin{proof}%[Proof of Lemma \ref{lm:ind2}]
Expanding the moments, we have
\begin{align*}
 \left|\E \prod_{v=1}^K T_{G_{n, \varepsilon}^+}(v)^{s_v} -\prod_{v=1}^K \E T_{G_{n, \varepsilon}^+}(v)^{s_v}\right| & = \sum_{\Gamma}  \left|\E \prod_{v=1}^K \prod_{j=1}^{s_v} \bm 1 \{X_v=X_{b_j(v)}\}-\prod_{v=1}^K \E \prod_{j=1}^{s_v} \bm1\{X_v=X_{b_j(v)}\}\right| \nonumber \\ 
& =\sum_{\Gamma}\left|\beta\left(\bigcup_{v=1}^K H(v)\right)-\prod_{v=1}^K \beta(H(v))\right|
\end{align*}
where 
\begin{itemize} 
\item[--]$\Gamma$ is the collection of all possible choices of $b_j(v)\in N_{G_{n, \varepsilon}}(v)$, for $j\in [s_v]$ and $v\in [K]$; and 
\item[--] $H(v)$ denotes the simple graph formed by union of all the edges $(v, b_j(v))$, for $j \in [s_v]$. Note that $H(v)$ is isomorphic to a star graph, for every $v\in [K]$.
\end{itemize}

If $\bigcup_{v=1}^K H(v)$ is a forest, then the collection of random variables $\{\bm1\{X_v=X_{b_j(v)}, j \in [s_v], v\in [K]\}$ are mutually independent, and so, $\beta(\bigcup_{v=1}^K H(v))=\prod_{v=1}^K \beta(H(v))$. Thus, without loss of generality,  assume that $\bigcup_{v=1}^K H(v)$ is not a forest, that is, it contains a cycle. Then denoting $\cH_m$ to be the set of unlabelled graphs with $m$ vertices and $s:=\sum_{v=1}^K s_v$, using Lemma \ref{lm:beta_H}  gives 
\begin{align}\label{eq:momentdiffT}
\left|\E \prod_{v=1}^K T_{G_{n, \varepsilon}^+}(v)^{s_v}-\prod_{v=1}^K \E T_{G_{n, \varepsilon}^+}(v)^{s_v}\right|\lesssim& \sum_{m=2}^{2s}\sum_{\substack{H\in \cH_m \\ H\text{ contains a cycle }}}\sum_{\Gamma: \bigcup_{v=1}^K H(v)\simeq H} \beta\left(\bigcup_{v=1}^K H(v)\right) \nonumber \\
=&\sum_{m=2}^{2s}\sum_{\substack{H\in \cH_m \\ H\text{ contains a cycle }}} N(H,G_{n, \varepsilon}^+[K]) \beta(H) \nonumber \\
=&\sum_{m=2}^{2s} \sum_{\substack{H\in \cH_m \\ H\text{ contains a cycle }}} \frac{N(H,G_{n, \varepsilon}^+[K])}{c_n^{|V(H)|-\nu(H)}}.
\end{align}

Now, fix $H \in \cH_m$ with connected components $H_1, H_2, \ldots, H_{\nu(H)}$, and assume without loss of generality that $H_1$ contains a cycle of length $g\ge 3$. Invoking \cite[Lemma 2.3]{BDM} gives,  $$N(H_1,G_{n, \varepsilon}^+[K])\lesssim |E(G_{n, \varepsilon}^+[K])|^{|V(H_1)|-g/2}\lesssim (K\Delta(G_n))^{|V(H_1)|-g/2},$$
where the last inequality uses $|E(G_{n, \varepsilon}^+[K]|\le K\Delta(G_n)$.  Also, by \cite[Lemma 2.3]{BDM}, for $j\ge 2$, 
$$N(H_j,G_{n, \varepsilon}^+[K])\lesssim |E(G_{n, \varepsilon}^+[K])|^{|V(H_j)|-1}\le (K\Delta(G_n))^{|V(H_j)|-1}.$$
Taking a product over $j$ and using $\Delta(G_n) = O(c_n)$, gives 
$$N(H,G_{n, \varepsilon}^+[K])\le \prod_{j=1}^{\nu(H)} N(H_j, G_n)\lesssim K^{|V(H)|-\nu(H)|}c_n^{|V(H)|-g/2},$$
which implies $\limsup_{n\rightarrow\infty}\frac{N(H,G_{n, \varepsilon}^+[K])}{c_n^{|V(H)|-\nu(H)}}=0$, as $g\ge 3$. Since the sum in \eqref{eq:momentdiffT} is finite (does not depend on $n,\varepsilon$), the  conclusion in \eqref{eq:moment_ind} follows.

Moreover, since $T_{G_{n, \varepsilon}^+}(v) \rightarrow \dPois(\theta_v)$ in distribution and in moments, \eqref{eq:moment_ind}  implies that 
$$\lim_{n\rightarrow\infty}\left|\E\left( \prod_{v=1}^\eta T_{G_{n, \varepsilon}^+}(v)^{s_v}\right)- \prod_{v=1}^{\eta} \E \dPois(\theta_v)^{s_v} \right|.$$
This implies, as the Poisson distribution is uniquely determined by its moments, $$(T_{G_{n, \varepsilon}^+}(1), T_{G_{n, \varepsilon}^+}(2), \ldots, T_{G_{n, \varepsilon}^+}(\eta))\rightarrow (T_1, T_2, \ldots, T_\eta),$$ as $n \rightarrow \infty$, in distribution and in moments, where $T_1, T_2, \ldots, T_\eta$ are independent $\dPois(\theta_1), \dPois(\theta_2)$, $\ldots, \dPois(\theta_\eta)$, respectively. Finally, recalling \eqref{eq:T1eps} and by the continuous mapping theorem $T_+(K_{1, r}, G_{n, \varepsilon}^{+})=\sum_{v=1}^\eta {T_{G_{n, \varepsilon}^+}(v) \choose r} \rightarrow \sum_{v=1}^\eta {T_v \choose r}$ in distribution and in moments, as $n \rightarrow \infty$. \end{proof}

\subsection{Contribution from $G_{n, \varepsilon}^-$} 
\label{sec:pfm3}

In this section we derive the limiting distribution of $T(K_{1,r},G_{n, \varepsilon}^{-})$, by invoking \cite[Theorem 2.1]{BMM}, which gives conditions under which the number of monochromatic subgraphs (in particular monochromatic stars) converges to a linear combination of Poisson variables.

\begin{lem}\label{lm:triangleapprox} As $n \rightarrow \infty$ followed by $\varepsilon \rightarrow 0$, 
$$T(K_{1,r},G_{n, \varepsilon}^{-}) \rightarrow \sum_{k=1}^{r+1} k Z_k,$$ in distribution and in  moments, where $Z_1, Z_2, \ldots, Z_{r+1}$ are independent  $\dPois(\lambda_1-\frac{1}{r!}\sum_{u=1}^\infty \theta_u^r)$, $\dPois(\lambda_2), \ldots \dPois(\lambda_{r+1})$, respectively.
\end{lem}

\subsubsection*{\textbf{Proof of Lemma \ref{lm:triangleapprox}}} We will prove this result by invoking \cite[Theorem 2.1]{BMM}. To begin with, let $F$ be a graph formed by union of two isomorphic copies of $K_{1, r}$, such that $|V(F)| > r+1$. Then $F$ is connected, and
\begin{align*}
N(F, G_{n, \varepsilon}^{-}) \lesssim N(K_{1, r}, G_{n, \varepsilon}^{-}) \cdot \Delta(G_n)^{|V(F)|-r-1} 
& \leq N(K_{1, r}, G_n) \cdot (\varepsilon c_n)^{|V(F)|-r-1} \nonumber \\ 
&=\varepsilon^{|V(F)|-r-1}  c_n^{|V(F)|-1}.
\end{align*}
Therefore, $\frac{1}{ c_n^{|V(F)|-1}} N(F, G_{n, \varepsilon}^{-})=o(1)$,  $n \rightarrow \infty$ followed by $\varepsilon \rightarrow 0$, when  $|V(F)| > r+1$.

It remains to consider super-graphs  $F \supseteq K_{1, r}$ with $|V(F)|=r+1$. Recalling $\sC_{r,k} := \{F \supseteq K_{1, r}: |V(F)|=r+1 \text{ and } N(K_{1, r},F)=k\}$, we have the following lemma.

\begin{lem}\label{lm:indF} For any $F \in \sC_{r, k}$, with $k \in [2, r+1]$, $N_{\mathrm{ind}}(F, G_{n, \varepsilon})=N_{\mathrm{ind}}(F, G_{n, \varepsilon}^-)+o(c_n^r)$, as $n \rightarrow \infty$ followed by $\varepsilon \rightarrow 0$. 
\end{lem}

\begin{proof} Let $k \in [2, r+1]$ and suppose $F \in \sC_{r, k}$  is an induced subgraph of $G_{n, \varepsilon}$, such that $V(F)$ is not completely contained in $V(G_{n, \varepsilon}^-)$. Then, since $F$ has at least two vertices of degree $r$ and any two degree $r$ vertices must be neighbors, the vertices of $F$ can be spanned by a $r$-star whose central vertex is in $N_{G_{n, \varepsilon}}(V_\varepsilon(G_n))$. Therefore, the difference $N_{\mathrm{ind}}(F, G_{n, \varepsilon})-N_{\mathrm{ind}}(F, G_{n, \varepsilon}^-)$ is bounded above by (up to constants depending only on $r$) 
\begin{align}\label{eq:smallvertexbound}
\sum_{v \notin V_\varepsilon(G_n)}\sum_{u_1\in V_\varepsilon(G_n)} \sum_{\bm u \in {V(G_n) \choose r-1}}  a_{v u_1}(G_n)a_v(\bm u, G_{n}),
\end{align} 
which is $o(c_n^r)$ (from the proof of Lemma \ref{lm:TS2}). 
\end{proof}

Using the above lemma and $N_{\mathrm{ind}}(F, G_n)=N_{\mathrm{ind}}(F, G_{n, \varepsilon})+o(c_n^r)$ (by Lemma \ref{lm:Geps}), it follows that, for $k \in [2, r+1]$, 
$$\lim_{\varepsilon \rightarrow 0} \lim_{n \rightarrow \infty}  \frac{\sum_{F\in \sC_{r,k}} N_{\mathrm{ind}}(F, G_{n, \varepsilon}^-)}{c_n^{r}} = \lim_{\varepsilon \rightarrow 0} \lim_{n \rightarrow \infty}  \frac{\sum_{F\in \sC_{r,k}} N_{\mathrm{ind}}(F, G_n)}{c_n^{r}} =\lambda_k,$$
where the last equality uses \eqref{eq:Finduced}. 

It remains to consider the case $k=1$. To begin with, observe that for any graph $G$, 
\begin{align}\label{eq:NHsum}
N(K_{1, r}, G)=\sum_{k=1}^{r+1} \sum_{F \in \sC_{r, k}}  k N_{\mathrm{ind}}(F, G).
\end{align}	
Moreover, using Lemma \ref{lm:Geps} and \eqref{eq:m2starreduction} gives
%since the number of $r$-stars in $G_{n, \varepsilon}$ with central vertex in $N_{G_{n, \varepsilon}}(V_\varepsilon(G_n))$ is bounded above by \eqref{eq:smallvertexbound} which is $o(c_n^r)$, 
\begin{align*}
N(K_{1, r}, G_{n, \varepsilon}^-)%&=N(K_{1, r}, G_{n, \varepsilon})-N(K_{1, r}, G_{n, \varepsilon}^+)+o(c_n^r) \\
&=N(K_{1, r}, G_n)-\sum_{v=1}^\eta {d_v \choose r}+o(c_n^r).% \tag*{(using Lemma \ref{lm:Geps}).}
\end{align*}
Now, using this and \eqref{eq:NHsum} with $G=G_{n, \varepsilon}^-$ gives 
\begin{align*} 
\frac{\sum_{F\in \sC_{r,1}} N_{\mathrm{ind}}(F, G_{n, \varepsilon}^-)}{c_n^{r}} & = \frac{N(K_{1, r}, G_n)}{c_n^r} - \frac{1}{c_n^r} \sum_{v=1}^\eta {d_v \choose r} - \sum_{k=2}^{r+1} k \sum_{F \in \sC_{r, k}}  \frac{N_{\mathrm{ind}}(F, G_{n, \varepsilon}^-)}{c_n^r} + o(1) \\
& \rightarrow  \sum_{k=1}^{r+1} k \lambda_k  -\sum_{u=1}^\infty \frac{\theta_u^r}{r!}-  \sum_{k=2}^{r+1} k \lambda_k  \tag*{(using \eqref{eq:NHsum} with $G=G_n$ and \eqref{eq:Finduced})}  \\
&= \lambda_1 -\sum_{u=1}^\infty \frac{\theta_u^r}{r!}, 
\end{align*}
as $n \rightarrow \infty$ followed by $\varepsilon \rightarrow 0$. Then by \cite[Theorem 2.1]{BMM}, we have $T(K_{1,r},G_{n, \varepsilon}^{-}) \dto \sum_{k=1}^{r+1} k Z_k$, where $Z_1, Z_2, \ldots, Z_{r+1}$ are as in the statement of the lemma.

The convergence in moments is a consequence of uniform integrability as $\E(T(K_{1, r}, G_{n, \varepsilon}^-)) \leq \E T(K_{1, r}, G_n)^r=O_r(1)$ for every fixed integer $r \geq 1$ \cite[Theorem 1.2]{BDM}.

\subsection{Completing the Proof of Theorem \ref{th:2star}}
\label{sec:pf2star}
%To complete the proof, consider first the more interesting case $\lim_{\varepsilon\rightarrow 0}\eta(\varepsilon)=+\infty$. 

To begin use Lemma \ref{lm:TS2} to note that it suffices to find the limiting distribution of 
$$\sum_{v=1}^\eta{T_{G_{n,\varepsilon}^+}(v)\choose r}+T(K_{1,r},G_{n,\varepsilon}^-),$$
under the double limit as $n\rightarrow\infty$ followed by $\varepsilon\rightarrow 0$. Fix an integer $M\ge 1$ and write the above random variable as
$$\sum_{v=1}^M{T_{G_{n,\varepsilon}^+}(v)\choose r}+\sum_{v=M+1}^\eta{T_{G_{n,\varepsilon}^+}(v)\choose r}+T(K_{1,r},G_{n,\varepsilon}^-).$$
Under the double limit the random vector 
$$\Big(T_{G_{n,\varepsilon}^+}(1),\cdots, T_{G_{n,\varepsilon}^+}(M), T(K_{1,r},G_{n,\varepsilon}^-)\Big)\dto \left(T_1,\cdots,T_M,\sum_{k=1}^{r+1}kZ_k\right),$$
 by invoking Lemmas \ref{lm:ind}, \ref{lm:ind2} and \ref{lm:triangleapprox}. 
By continuous mapping theorem this gives
$$\sum_{v=1}^M{T_{G_{n,\varepsilon}^+}(v)\choose r}+T(K_{1,r},G_{n,\varepsilon}^-)\stackrel{D}{\rightarrow} \sum_{v=1}^M{T_v\choose r}+\sum_{k=1}^{r+1}kZ_k,$$
the RHS of which on letting $p\rightarrow\infty$ converges in distribution to $\sum_{v=1}^\infty {T_v\choose r}+\sum_{k=1}^{r+1}kZ_k$. It thus suffices to show that
$$\lim_{M\rightarrow\infty}\lim_{\varepsilon\rightarrow0}\lim_{n\rightarrow\infty} \sum_{v=M+1}^{\eta}\E{T_{G_{n,\varepsilon}^+}(v)\choose r}=0.$$
The LHS above is bounded above by $ \sum_{v=M+1}^{\eta} \frac{1}{r!}\frac{d_v^r}{c_n^r}$, which on letting $n\rightarrow\infty$ followed by $\varepsilon\rightarrow0$ gives $\frac{1}{r!}\sum_{v=M+1}^\infty \theta_v^r$.  This converges to $0$ as $M\rightarrow\infty$, as $\sum_{v=1}^\infty \theta_v^r<\infty$, as noted in the proof of Lemma \ref{lm:TS2}. (Note that if $\lim_{\varepsilon\rightarrow 0}\eta(\varepsilon):=L<\infty$, then the term $\sum_{v=L+1}^{\eta}{T_{G_{n,\varepsilon}^+}(v)\choose r}+T(K_{1,r},G_{n,\varepsilon}^-)$ vanishes, thus simplifying the proof. )

Finally, the convergence in moments is a consequence of uniform integrability as all moments of $T(K_{1, r}, G_n)$ are bounded: that is, $\E T(K_{1, r}, G_n)^r=O_r(1)$ for every fixed integer $r \geq 1$ (this follows from the proof of \cite[Theorem 1.2]{BDM}).

To prove the converse, invoking Proposition \ref{ppn:bmm} we can assume, without loss of generality, that $N(K_{1,r},G_n)=O(c_n^r)$. This in turn implies that for every graph $F$ on $r+1$ vertices which is a super graph of $K_{1,r}$ we have $N_{\text{ind}}(F,G_n)=O(c_n^r)$.  Thus by passing to a subsequence, assume that $N_{\text{ind}}(F,G_n)/c_n^r$ converges for every $F$ which is a super graph of $K_{1,r}$. This implies existence of the limits in \eqref{eq:Finduced}. Finally, using \eqref{eq:edges} we have $\max_{v\in V(G_n)}d_v=O(c_n)$, and so the infinite tuple $\{d_v/c_n\}_{v\ge 1}$ is an element of $[0,K]^\N$ for some $K$ fixed. Since $[0,K]^\N$ is compact in product topology, there is a further subsequence along which $d_v/c_n$ converges for every $v\ge 1$ simultaneously. Thus, moving to a subsequence, we can assume that $d_v/c_n$ converges to $\theta_v$ for every $v$. Invoking the sufficiency part of the theorem gives that $T(K_{1,r},G_n)$ converges in distribution to a random variable of the desired form, completing the proof.

\subsection{Proof of Corollary \ref{cor:2star}}
\label{sec:cor2starpf}

The proof of $(a)\Rightarrow (b)$ is immediate from Theorem \ref{th:2star}, so it suffices to prove $(b)\Rightarrow (a)$. To this end,  note that $T(K_{1,r},G_n) \dto \sum_{k=1}^{r+1} k Z_k$ implies that \eqref{eq:2starmean} holds (Proposition \ref{ppn:bmm}). Thus, by a similar argument which was used to prove the converse of Theorem \ref{th:2star}, it follows that along a subsequence the limits $\lim_{n\rightarrow\infty}\frac{1}{c_n^r} N_{\text{ind}}(F,G_n)$ exist for all super graphs $F$ of $K_{1,r}$ on $r+1$ vertices, and so, for $k \in [1, r+1]$, 
$$\lambda'_k:=\lim_{n\rightarrow\infty}\sum_{F\in \cC_{r,k}}\frac{N_{\text{ind}}(F,G_n)}{c_n^r} $$ is well defined. Then, as before, by passing to another subsequence the limits $\theta'_v:=\lim_{n\rightarrow\infty}\frac{d_v}{c_n}$ exist for every $v\ge 1$, and by the if part of Theorem \ref{th:2star} along this subsequence, 
$$T(K_{1,r},G_n)\stackrel{d}{\rightarrow}\sum_{v=1}^\infty{T_v'\choose r}+\sum_{k=1}^{r+1}k Z_k',$$
where $\{T_v'\}_{v\ge 1}$ and $\{Z_k'\}_{1\le k\le r+1}$ are mutually independent, and $T_1', T_2', \ldots, $ are independent $\dPois(\theta_1')$, $\dPois(\theta_2'), \ldots $, respectively, and $Z_1', Z_2', \ldots, Z_{r+1}'$ are independent  $\dPois(\lambda_1'-\frac{1}{r!}\sum_{u=1}^\infty (\theta_u')^r), \dPois(\lambda_2')$, $\ldots, \dPois(\lambda_{r+1}')$, respectively.

However, since $T(K_{1,r},G_n)$ converges in distribution to $\sum_{k=1}^{r+1}kZ_k$ which has finite exponential moment everywhere, it follows that $\theta_v'=0$ for all $v\ge 1$, and consequently, the maximum degree $\Delta(G_n)=o(c_n)$. This also gives $$\sum_{k=1}^{r+1}kZ_k\stackrel{D}{=}\sum_{k=1}^{r+1}kZ'_k,$$ 
and so the corresponding probability generating functions must match, that is, 
$$\prod_{k=1}^{r+1}e^{\lambda_k(s^k-1)}=\prod_{k=1}^{r+1}e^{\lambda_k'(s^k-1)}, \quad \text{for all } s\in (0, 1).$$
This implies, $\sum_{k=1}^{r+1}\lambda_k(s^k-1)=\sum_{k=1}^{r+1}\lambda_k'(s^k-1)$, for all $s\in (0,1)$, and so the corresponding coefficients must be equal, giving $\lambda_k=\lambda_k'$. Therefore, every sub sequential limit of $\sum_{F\in \cC_{r,k}}\frac{N_{\text{ind}}(K_{1,r},G_n)}{c_n^r}$ equal $\lambda_k$,  for $k \in [1, r+1]$, hence, \eqref{eq:Finduced} holds.

\section{Examples}
\label{sec:2starexample}

In this section we apply Theorem \ref{th:2star} to different deterministic and random graph models, and
determine the specific nature of the limiting distribution.

\begin{example}(Disjoint Union of Stars) The proof of Theorem \ref{th:2star} shows that the quadratic term in the limiting distribution of $T(K_{1,r}, G_n)$ appears due to the $r$-stars incident on vertices with degree $\Theta(c_n)$. This can be seen when $G_n$ is a disjoint union of star graphs.

\begin{itemize}
\item

To begin with suppose $G_n=K_{1, n}$ is the $n$-star. Then $N(K_{1, r}, K_{1, n})={n \choose r}$, and if  we color $K_{1, n}$ with $c_n$ colors such that $n/c_n \rightarrow 1$, then $\E(T(K_{1, r}, G_n))=\frac{1}{r!}$. Note that the maximum degree $d_{(1)}=n$, which implies $\theta_1=1$. Moreover, $d_{(2)}=1$, which implies $\theta_v=0$, for all $v \geq 2$. Therefore, by Theorem \ref{th:2star}, 
$$T(K_{1, r}, G_n)  \dto {T_1\choose r},$$
where $T_1\sim \dPois(1)$. (Note that the graph $G_{n, \varepsilon}^{-}$ is empty in this case.)

\item
Next, consider $G_n$ to be the disjoint union of the following stars: $K_{1, \lfloor na_1 \rfloor}, K_{1, \lfloor na_2\rfloor}, \ldots, K_{1, \lfloor na_n \rfloor}$, such that $\sum_{s=1}^\infty a_s^r  < \infty$. In this case, $N(K_{1, r}, G_n)=\sum_{s=1}^n{\lfloor na_s \rfloor \choose r}\sim \frac{n^r}{r!}\sum_{s=1}^na_s^r$. If  $G_n$ is colored with $c_n$ colors such that $n/c_n\rightarrow 1$, then $\E(T(K_{1, r}, G_n))\rightarrow \frac{1}{r!}\sum_{s=1}^\infty a_s^r$.  Also, $d_{(v)}=\lfloor na_v \rfloor$, which implies $\theta_v= a_v,$ for $v \geq 1$. This implies, by Theorem \ref{th:2star},  $$T(K_{1, r}, G_n) \dto \sum_{s=1}^\infty {T_s\choose r},$$ where $T_s\sim \dPois(a_s)$ and $T_1, T_2, \ldots$ are independent. Here, the linear terms linear in Poisson do not contribute, as $G_{n, \varepsilon}^-$ is empty, and $\E T(K_{1, r},G_n)\sim \frac{1}{r!}\sum_{v=1}^\infty \theta_v^r$.

\item Finally, consider $G_n$ to be the disjoint union of the following stars: $$K_{1, \lfloor na_1+n^{\frac{r-1}{r}} \rfloor}, K_{1, \lfloor na_2+n^{\frac{r-1}{r}} \rfloor}, \ldots, K_{1, \lfloor na_n+n^{\frac{r-1}{r}} \rfloor}.$$ In this case, $$N(K_{1, r}, G_n)=\sum_{s=1}^n{\lfloor na_s+n^{\frac{r-1}{r}}\rfloor \choose r}\sim \frac{n^r}{r!}+\frac{n^r}{r!}\sum_{s=1}^na_s^r,$$
since $\sum_{s=1}^n a_s^k =o(n^{1-\frac{k}{r}})$, for $1 \leq k < r$ (see Observation \ref{obs:serieskr} below).
If  $G_n$ is colored with $c_n$ colors such that $n/c_n\rightarrow1$, then $\E(T(K_{1, r}, G_n))\rightarrow \frac{1}{r!}\left(1+\sum_{s=1}^\infty a_s^r \right)$.  Also, $d_{(v)}=\lfloor na_v  + n^{\frac{r-1}{r}}  \rfloor$, which implies $\theta_v=a_v,$ for $v \geq 1$, and so
 Theorem \ref{th:2star} gives $$T(K_{1, r}, G_n) \dto \sum_{s=1}^\infty {T_s\choose r}+Z,$$ where $T_s\sim \dPois( a_s)$ and $T_1, T_2, \ldots$ are independent, and $Z\sim \dPois(\frac{1}{r!})$ independent of $\{T_s\}_{s\ge 1}$.

\end{itemize}
\end{example}

\begin{obs}\label{obs:serieskr} If $\{a_s\}_{s\geq 1}$ is a sequence of non-negative real numbers such that $\sum_{s=1}^\infty a_s^r < \infty$ then $\sum_{s=1}^n a_s^k =o(n^{1-\frac{k}{r}})$, for $1 \leq k < r$. 
\end{obs}

\begin{proof} Fixing $\varepsilon>0$ and a positive integer $N\ge 1$ we get
\begin{align*}
\sum_{s=1}^n a_s^k=&\sum_{s=1}^Na_s^k+\sum_{s=N+1}^n a_s^k \bm 1\{a_s\le \varepsilon n^{-\frac{1}{r}}\}+\sum_{s=N+1}^n a_s^k \bm 1\{a_s>\varepsilon n^{-\frac{1}{r}}\}\\
\le &\sum_{s=1}^Na_s^k+\varepsilon^k n^{1-\frac{k}{r}}+\frac{n^{1-\frac{r}{k}}}{\varepsilon^{r-k}}\sum_{s=N+1}^\infty  a_s^r.
\end{align*}
On dividing by $n^{1-\frac{k}{r}}$ and letting $n\rightarrow\infty$, the first term goes to $0$ as it is a finite sum, and, therefore, 
$$\limsup_{n\rightarrow\infty}\frac{\sum_{s=1}^ka_s^r}{n^{1-\frac{k}{r}}}\le \varepsilon^k+\frac{1}{\varepsilon^{r-k}}\sum_{s=N+1}^\infty a_s^r. $$
The desired conclusion now follows on letting $N\rightarrow\infty$ followed by $\varepsilon\rightarrow 0$, on noting that $\sum_{s=1}^\infty a_s^r<\infty$. 
\end{proof}

Next, we see examples where there are no vertices of high degree, in which case, the quadratic term vanishes (Corollary \ref{cor:2star}).

\begin{example}(Regular Graphs) Let $G_n $ be a $d$-regular graph. In this case, $N(K_{1, r}, G_n)=n {d \choose r}$. Consider uniformly coloring the graph with $c_n$ colors such that $\frac{1}{c_n^r} n {d \choose r} \rightarrow \lambda$. In this case,  $\Delta(G_n)=\max_{v\in V(G_n)} d_v=d=o(c_n)$. Therefore, by Corollary \ref{cor:2star}, $T(K_{1, r}, G_n) \dto \sum_{k=1}^{r+1} k Z_k$, where $Z_1, Z_2, \ldots, Z_{r+1}$ are independent $\dPois(\lambda_1), \dPois(\lambda_2), \ldots, \dPois(\lambda_{r+1})$ (recall \eqref{eq:Finduced}). (Note that $\sum_{k=1}^{r+1} k\lambda_k=\lambda$.) The limit simplifies in special cases: 

\begin{itemize}
%\item[--] $G_n=C_n$ the cycle on $n$-vertices. In this case $N(K_3, G_n)=0$, for $n \geq 3$, and therefore, $T(K_{1, 2}, G_n) \dto \dPois(\nu)$.

\item[--] $G_n=K_{n, n}$, the regular bipartite graph. Since, bipartite graphs are triangle-free, $N_{\mathrm{ind}}(F, G_n)=0$, for any super-graph $F$ of $K_{1, r}$ with $|V(F)|=r+1$. This implies $\lambda_k=0$, for $2 \leq k \leq r+1$, and $\lambda_1=\lambda$, and $T(K_{1, r}, K_{n, n}) \dto \dPois(\lambda)$.

\item[--] $G_n=K_n$, the complete graph on $n$ vertices. In this case, any induced graph on $r+1$ vertices is isomorphic to $K_{r+1}$. This implies $\lambda_k=0$, for $1 \leq k \leq r$ and $\lambda_{r+1}=\frac{\lambda}{r+1}$, and $T(K_{1, r}, K_n) \dto (r+1) Z_{r+1}$, where $Z_{r+1}\sim \dPois(\frac{\lambda}{r+1})$.
\end{itemize}

\end{example}

Note that in all the above examples, the limiting distribution either involves only the quadratic part or only the linear part. It is easy to construct examples where both the components show up by taking disjoint unions (or connecting them with a few edges) of the graphs in the above examples, as shown below:

%%%%%%%%%%%%%%%%%%%%%%%%%%%%%%%%%%%%%%%%%%%%%%%%%%%%%%%%%%%%%%%%%%%%%%%%%%%%%%%%
\begin{figure*}[h]
\centering
\begin{minipage}[c]{1.0\textwidth}
\centering
\includegraphics[width=3.85in]
    {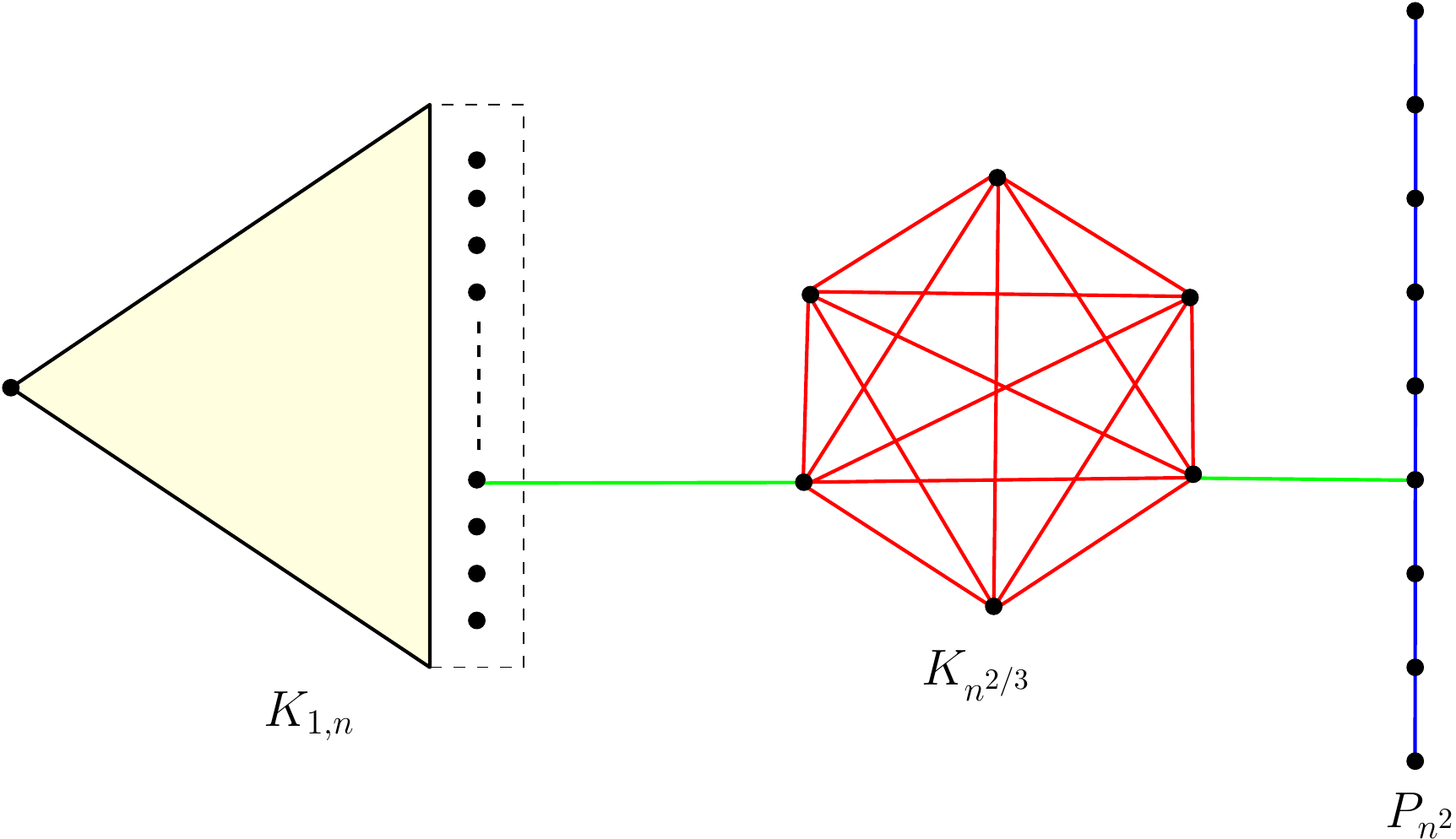}\\
%\small{(a)}
\end{minipage}
\caption{\small{Illustration for Example \ref{ex:graph_disjoint_union}.}}
\vspace{-0.1in}
\label{fig:Graph_Disjoint_Union}
\end{figure*}
%%%%%%%%%%%%%%%%%%%%%%%%%%%%%%%%%%%%%%%%%%%%%%%%%%%%%%%%%%%%%%%%%%%%%%%%%%%%%%%%%%%%%%%%%%%%%%%%%%%%%%%%%%

\begin{example}\label{ex:graph_disjoint_union} 
Let $G_n$ be the graph in Figure \ref{fig:Graph_Disjoint_Union}. Note that it has three parts, a $K_{1, n}$, where one of the leaves is connected by a single edge to a $K_{n^{2/3}}$, which is connected by a single edge to a path $P_{n^2}$. Consider coloring this graph by $c_n$ colors such that $c_n/n\rightarrow \kappa$. This implies 
$$\E(T(K_{1, 2}, G_n))=\frac{1}{c_n^2}N(K_{1, 2}, G_n)\sim \frac{{ n \choose 2}+3{\ceil{n^{2/3}}\choose 3}+n^2}{c_n^2}\rightarrow 2\kappa^2.$$
Next, note that $\Delta(G_n)=n$, which corresponds to the central vertex of the $K_{1, n}$. Therefore, $\theta_1=\kappa$. For every other vertex the degree is $o(n)$, which implies $\theta_v=0$, for all $v \geq 2$. Finally, since $N(K_3, G_n)={\ceil{n^{2/3}}\choose 3}$, $\nu:=\lim_{n\rightarrow \infty}\frac{1}{c_n^2}N(K_3, G_n)=\frac{\kappa^2}{6}$. Therefore, by Theorem \ref{th:2star}
$$T(K_{1, 2}, G_n) \dto {T_1\choose 2}+3Z_3+Z_1,$$ 
where $T_1\sim \dPois(\kappa)$, $Z_3 \sim \dPois(\frac{\kappa^2}{6})$, and $Z_1 \sim \dPois(\frac{\kappa^2}{2})$. 
\end{example}

\begin{remark}(Extension to random graphs) By a simple conditioning argument, Theorem \ref{th:2star} can be extended to random graphs by conditioning on the graph, under the assumption that the graph and its coloring are jointly independent (see \cite[Lemma 4.1]{BMM}). In this case, whenever the limits in \eqref{eq:2starmean} and \eqref{eq:degassumption} exist in probability, the limit \eqref{eq:2star} holds.  For example, when $G_n\sim G(n, p(n))$ is the Erd\H os-R\'enyi random graph, then the limiting distribution of $T(K_{1, r}, G_n)$ (when $c_n$ is chosen such that $\frac{1}{c_n^r}\E(N(K_{1, r}, G_n)) \rightarrow \lambda$) can be easily derived using Theorem \ref{th:2star}. In this case, depending on whether (a) $n^{\frac{r+1}{r}}p(n)\rightarrow O(1)$, (b) $p(n)\rightarrow 0, n^{\frac{r+1}{r}}p(n)\rightarrow\infty$, or (c) $p(n)=p\in (0, 1)$ is fixed, $T(K_{1, r}, G_n)$ converges to (a) zero in probability, or (b) $\dPois(\lambda)$, or (c) a linear combination of independent Poisson variables (see \cite[Theorem 1.3]{BMM} for details).
\end{remark}

\section{Conclusion and Open Problems}
\label{sec:conclusion}

This paper studies the limiting distribution of the number of monochromatic $r$-stars in a uniformly random coloring of a growing graph sequence. We provide a complete characterization of the limiting distribution of $T(K_{1,r}, G_n)$, in the regime where $\E(T(K_{1,r}, G_n))=\Theta(1)$.

It remains open to understand the limiting distribution of $T(K_{1,r}, G_n)$ when $\E(T(K_{1,r}, G_n))=\frac{1}{c_n^2} N(K_{1,r}, G_n)$ grows to infinity. For the case of monochromatic edges, \cite[Theorem 1.2]{BDM} showed that $T(K_2, G_n)$ (centered by the mean and scaled by the standard deviation) converges to $N(0, 1)$, whenever $\E(T(K_2, G_n))=\frac{1}{c_n} |E(G_n)| \rightarrow \infty$ such that $c_n \rightarrow \infty$. Error rates for the above CLT were obtained by Fang \cite{xiao}. 
It is natural to wonder whether this universality phenomenon extends to monochromatic $r$-stars, and more generally, to any fixed connected graph $H$.

On the other hand, when $\E(T(K_2, G_n)) \rightarrow \infty$ such that the number of colors $c_n=c$ is fixed, then $T(K_2, G_n)$ (after appropriate centering and scaling) is asymptotically normal if and only if its fourth moment converges to 3 \cite[Theorem 1.3]{BDM}. It would be interesting to explore whether this fourth-moment phenomenon extends to monochromatic $r$-stars.  
\\

\small{\noindent{\bf Acknowledgement:} The authors are indebted to Somabha Mukherjee for his careful comments on an earlier version of the manuscript, and Swastik Kopparty for helpful discussions.}

\end{document}